\documentclass[a4paper, 11pt]{article}
\usepackage{amsmath}
\usepackage{amssymb,esint}
\usepackage{amscd}
\usepackage{xspace}
\usepackage{fancyhdr}
\usepackage{color}
\usepackage{authblk}
\usepackage{srcltx} 

\newtheorem{theorem}{Theorem}[section]

\newtheorem{definition}[theorem]{Definition}

\newtheorem{lemma}[theorem]{Lemma}

\newtheorem{proposition}[theorem]{Proposition}
\newtheorem{remark}[theorem]{Remark}

\newenvironment{proof}[1][Proof]{\textbf{#1.} }{\hfill\rule{0.5em}{0.5em}}
{\catcode`\@=11\global\let\AddToReset=\@addtoreset
\AddToReset{equation}{section}

\AddToReset{theorem}{section}

\def\nc{\newcommand}

\def\div{\text{div}}

\nc\pa{\partial}

\nc\CC{\mathbb{C}}
\nc\RR{\mathbb{R}}
\nc\QQ{\mathbb{Q}}
\nc\ZZ{\mathbb{Z}}
\nc\NN{\mathbb{N}}

\begin{document}
\title{Good-$\lambda$ type bounds of quasilinear elliptic equations for the singular case}

\author{M. P. Tran \thanks{Faculty of Mathematics and Statistics, Ton Duc Thang University, Ho Chi Minh city, Vietnam; \texttt{tranminhphuong@tdtu.edu.vn}}}

\date{}  
\maketitle
\begin{abstract}
In this paper, we study the good-$\lambda$ type bounds for renormalized solutions to nonlinear elliptic problem:
\begin{align*}
\begin{cases}
-\div(A(x,\nabla u)) &= \mu \quad \text{in} \ \ \Omega, \\
u &=0 \quad \text{on} \ \ \partial \Omega.
\end{cases}
\end{align*}
where $\Omega \subset \mathbb{R}^n$, $\mu$ is a finite Radon measure and $A$ is a monotone Carath\'edory vector valued function defined on $W^{1,p}_0(\Omega)$. The operator $A$ satisfies growth and monotonicity conditions, and the $p$-capacity uniform thickness condition is imposed on $\mathbb{R}^n \setminus \Omega$, for the singular case $\frac{3n-2}{2n-1} < p \le 2- \frac{1}{n}$. In fact, the same good-$\lambda$ type estimates were also studied by Quoc-Hung Nguyen and Nguyen Cong Phuc. For instance, in \cite{55QH4,55QH5}, authors' method was also confined to the case of $\frac{3n-2}{2n-1} < p \le 2- \frac{1}{n}$ but under the assumption of $\Omega$ is the Reifenberg flat domain and the coefficients of $A$ have small BMO (bounded mean oscillation) semi-norms. Otherwise, the same problem was considered in \cite{55Ph0} in the regular case of $p>2-\frac{1}{n}$. In this paper, we extend their results, taking into account the case $\frac{3n-2}{2n-1} < p \le 2- \frac{1}{n}$ and without the hypothesis of Reifenberg flat domain on $\Omega$ and small BMO semi-norms of $A$. Moreover, in rest of this paper, we also give the proof of the boundedness property of maximal function on Lorentz spaces and also the global gradient estimates of solution.

\medskip

\medskip

\medskip

\noindent 

\medskip

\noindent Keywords: quasilinear elliptic equation; measure data; good-$\lambda$  inequality; capacity.

\end{abstract}   
                  
\section{Introduction}
\label{sec:intro}

In this paper, our goal is to obtain a ``good-$\lambda$ type bound of solutions (the renormalized solutions) to quasilinear elliptic equations with measure data:

\begin{equation}
\label{eq:elliptictype}
\begin{cases}
-\div(A(x,\nabla u)) &= \mu \quad \text{in} \ \ \Omega, \\
u &=0 \quad \text{on} \ \ \partial \Omega.
\end{cases}
\end{equation}
where the domain $\Omega$ is a bounded open subset of $\mathbb{R}^n$, $n \ge 2$, and $\mu$ is a finite signed Radon measure in $\Omega$. This kind of problem has been widely studied in \cite{11DMOP, AdP1, BMMP, BGO, BMP1, BMP2, BMMP2,GW} for the existence and uniqueness of renormalized solution. 
However, the global gradient estimates for solutions or gradient of solutions of such problem are still an open problem. Earlier, there are some literatures related to this work, on what follows, refer to \cite{Mi2,VHV}. The interior, exterior and boundary estimates on the gradient of renormalized solution to \eqref{eq:elliptictype} are also interesting needed to be studied.  In \cite{Mi2}, G. Mingione has proposed firstly the method of using the $1$-fractional maximal operator to get gradient estimates for $p>2$ and the interior case. Later, it has been developed by several authors in recent. In particular, Nguyen Cong Phuc has extended the approach this result up to the boundary case in \cite{55Ph0} for $p>2-\frac{1}{n}$, and later in \cite{55QH4, 55QH5}, authors gave a good global gradient estimates for solution to \eqref{eq:elliptictype} for $\frac{3n-2}{2n-1} < p \le 2- \frac{1}{n}$ by using the so-called Reifenberg flatness assumption on domain and in the weighted Lorentz space. Therein, the nonlinearity $A$ satisfies a smallest condition the BMO in the $x$-variable and the given results were proved the $L^{q,s}(\Omega)$ estimates of solution for all $q>0, 0<s \le \infty$. 

As far as we know, the uniform capacity density condition is weaker than the Reifenberg flatness condition. Therefore, in this paper, we formulate and establish a natural extension results concerning the global gradient estimates of solution to \eqref{eq:elliptictype} for $\frac{3n-2}{2n-1} < p \le 2- \frac{1}{n}$ under the $p$-capacity uniform thickness condition on domain $\Omega$. The nonlinearity $A$ here is a Carath\'edory vector valued function defined on $W^{1,p}_0(\Omega)$. The operator $A$ satisfies growth and monotonicity conditions: there holds
\begin{align*}
\left| A(x,\xi) \right| &\le \beta |\xi|^{p-1},\\
\langle A(x,\xi)-A(x,\eta), \xi - \eta \rangle &\ge \alpha \left( |\xi|^2 + |\eta|^2 \right)^{\frac{p-2}{2}}|\xi - \eta|^2,
\end{align*}
for every $(\xi,\eta) \in \mathbb{R}^n \times \mathbb{R}^n \setminus \{(0,0)\}$ and a.e. $x \in \mathbb{R}^n$, $\alpha$ and $\beta$ are positive constants. This operator and its properties are emphasized in Section \ref{sec:assump_opers}. In fact, there has been a research activity on the same gradient estimates in Morey and Lorentz spaces under these assumptions, using linear and nonlinear potentials to formulate the estimates. It can be found in \cite{KMi1} for the scalar case and \cite{KMi2} for the vectorial case.

In this paper, following the approaches developed by \cite{Mi3}, \cite{55Ph0}, our main effort is the proof of both results of boundedness property of maximal function and gradient estimates of solution to \eqref{eq:elliptictype} on $L^{q,s}(\Omega)$, for the singular case $\frac{3n-2}{2n-1} < p \le 2 - \frac{1}{n}$, stated in following theorems \ref{theo:lambda_estimate} and \ref{theolorentz_estimate}, respectively.

\begin{theorem}
	\label{theo:lambda_estimate}
	Let $\frac{3n-2}{2n-1}<p \le 2-\frac{1}{n}$ and suppose that $\Omega \subset \mathbb{R}^n$ is a bounded domain whose complement satisfies a $p$-capacity uniform thickness condition with constants $c_0, r_0>0$. Let $\mu \in \mathfrak{M}_b(\Omega)$ and $Q = B_{\text{diam}(\Omega)}(x_0)$, where $x_0$ is fixed in $\Omega$. \\
	Then, for  any renormalized solution $u$ to \eqref{eq:elliptictype} with given measure data $\mu$, there exist $\Theta = \Theta(n,p,\alpha,\beta,c_0)>p$, $\varepsilon_0 = \varepsilon_0(n,p,\alpha,\beta) \in (0,1)$ and constant $C = C(n,p,\alpha,\beta,\sigma,c_0,diam(\Omega)/{r_0})>0$ such that the following estimate
	\begin{align}
	\label{eq:mainlambda}
	&\left|\{({\bf M}(|\nabla u|^{\gamma_0}))^{1/\gamma_0}>\varepsilon^{-\frac{1}{\Theta}}\lambda, (\mathbf{M}_1(\mu))^{\frac{1}{p-1}}\le \varepsilon^{\frac{1}{(p-1)\gamma_0}}\lambda \}\cap Q \right| \nonumber\\
	&\qquad\leq C \varepsilon \left|\{ ({\bf M}(|\nabla u|^{\gamma_0}))^{1/\gamma_0}> \lambda\}\cap Q \right|,
	\end{align}
	holds for any $\lambda>0, \varepsilon\in (0,\varepsilon_0)$, and for some $\gamma_0 \in \left(\frac{2-p}{2},\frac{(p-1)n}{n-1} \right)$.
\end{theorem}

One notices that $\mathfrak{M}_b(\Omega)$ in this theorem stands for the Radon measure on $\Omega$ with bounded total variation, would be introduced in Section \ref{sec:assump_meas} later, and the denotation $diam(\Omega)$ is the diameter of a set $\Omega$ defined as:
\begin{align*}
diam(\Omega) = \sup\{d(x,y) | x,y \in \Omega\},
\end{align*}  
and in what follows, operators $\mathbf{M}, \mathbf{M}_1$ are introduced later in Section \ref{sec:othersdefs}.

\begin{theorem} \label{theolorentz_estimate} 
		Let $\frac{3n-2}{2n-1}<p \le 2-\frac{1}{n}$ and suppose that $\Omega \subset \mathbb{R}^n$ is a bounded domain whose complement satisfies a $p$-capacity uniform thickness condition with constants $c_0, r_0>0$. Let $\mu \in \mathfrak{M}_b(\Omega)$ and $Q = B_{\text{diam}(\Omega)}(x_0)$, where $x_0$ is fixed in $\Omega$. Then there exists $\Theta = \Theta(n,p,\alpha,\beta,c_0)>p$ such that for any renormalized solution $u$ to \eqref{eq:elliptictype} with given measure data $\mu$, $0<q<\Theta$ and $0<s\leq \infty$  it gives   
	\begin{equation*}
	\|\nabla u\|_{L^{q,s}(\Omega)}\leq C \|[\mathbf{M}_1(\mu)]^{\frac{1}{p-1}}\|_{L^{q,s}(Q)}.
	\end{equation*} 
	Here the constant $C$ depends only  on $n,p,\alpha,\beta,q, s,c_0$ and $diam(\Omega)/r_0$.               
\end{theorem}

For the proofs of these above theorems \ref{theo:lambda_estimate} and \ref{theolorentz_estimate} in the present paper, it is possible to apply some results developed for quasilinear equations with given measure data, or linear/nonlinear potential and Calder\'on-Zygmund theories (see in \cite{bebo, BW1, 11DMOP, 55DuzaMing, Duzamin2, Mi2, 55QH3, 55Ph0, 55Ph2}), to some new comparison estimates in the singular case $\frac{3n-2}{2n-1}<p \le 2 - \frac{1}{n}$. 

The outline of paper is organized as follows. In the next Section \ref{sec:pre} we begin with some preliminaries about our notation and assumptions more precisely. Section \ref{sec:lems} is devoted to some important lemmas of local interior and boundary comparison estimates, that are necessary for main results. In Section \ref{sec:main} we complete the proofs of the main theorems \ref{theo:lambda_estimate} and theorem \ref{theolorentz_estimate} in our framework.

\section{Preliminaries}
\label{sec:pre}

First of all, let us recall a few preliminaries about the definitions and assumptions on our problem. In this paper, $\Omega$ is a bounded, open subset of $\mathbb{R}^n$, $n \ge 2$; and there is no smoothness is assumed on $\partial\Omega$. 

\subsection{Definitions of capacities}

We begin with the definition of $p$-capacity. Let $p$ and $p'$ be real numbers, such that $1 \le p \le n$ and $p'$ the Holder conjugate exponent of $p$, i.e., $1/p+1/p'=1$. The $p$-capacity $\text{cap}_p(B,\Omega)$ for any set $B \subseteq \Omega$ with respect to $\Omega$ is defined as following.

The $p$-capacity of any compact set $K \subset \Omega$ is defined as:
\begin{align*}
\text{cap}_p(K,\Omega) = \inf \left\{ \int_\Omega{|\nabla \varphi|^p dx}: \varphi \in C_c^\infty, \varphi \ge \chi_K \right\},
\end{align*}
where $\chi_K$ is the characteristic function of $K$. The $p$-capacity of any open subset $U \subseteq \Omega$ is then defined by:
\begin{align*}
\text{cap}_p(U,\Omega) = \sup \left\{ \text{cap}_p(K,\Omega), \quad K \ \text{compact}, \quad K \subseteq U \right\}.
\end{align*}

Consequently, the $p$-capacity of any subset $B \subseteq \Omega$ is defined by:
\begin{align*}
\text{cap}_p(B,\Omega) = \inf \left\{ \text{cap}_p(U,\Omega), \quad U \ \text{open}, \quad B \subseteq U \right\}.
\end{align*}

A function $u$ defined on $\Omega$ is said to be $\text{cap}_p$-quasi continuous if for every $\varepsilon>0$ there exists $B \subseteq \Omega$ with $\text{cap}_p(B,\Omega) < \varepsilon$ such that the restriction of $u$ to $\Omega\setminus B$ is continuous. It is well known that every  function in $W^{1,p}(\Omega)$ has a $\text{cap}_p$-quasi continuous representative, whose values are defined $\text{cap}_p$-quasi everywhere in $\Omega$, that is, up to a subset of $\Omega$ of zero $p$-capacity. When we are dealing with the pointwise values of a function $u \in W^{1,p}(\Omega)$, for every subset $B$ of $\Omega$ we have:
\begin{align*}
\text{cap}_p (B,\Omega) = \inf \left\{\int_\Omega{|\nabla v|^p dx}: v \in W^{1,p}(\Omega) \right\},
\end{align*}
where $v=1 \ \text{cap}_p$-quasi everywhere on $B$, and $v \ge 0  \ \text{cap}_p$-quasi everywhere on $\Omega$.

\begin{definition}
\label{def:pcapuni}
By a capacity density condition on $\Omega$, we mention \emph{the $p-$capacity uniform thickness condition (with constants $r_0, c_0>0$) imposed on $\mathbb{R}^n \setminus \Omega$}. That is, there exist constants $c_0, r_0>0$ such that for all $0<t \le r_0$ and all $x \in \mathbb{R}^n \setminus \Omega$:
\begin{align}
\label{eq:capuni}
\text{cap}_p(\overline{B_t(x)} \cap \mathbb{R}^n \setminus \Omega, B_{2t}(x)) \ge c_0 \text{cap}_p(\overline{B_t(x)},B_{2t}(x)).
\end{align}
\end{definition}

Note that the domains satisfying \eqref{eq:capuni} include those with Lipschitz boundaries or even those that satisfy a uniform exterior corkscrew condition, means that there exist constants $c_0, r_0>0$ such that for all $0<t \le r_0$ and all $x \in \mathbb{R}^n \setminus \Omega$, there is $y \in B_t(x)$ such that $B_{t/c_0}(y) \subset \mathbb{R}^n\setminus\Omega$.

\subsection{Assumptions on measures}
\label{sec:assump_meas}
Firstly, we define $\mathfrak{M}_b(\Omega)$ as the space of all Radon measures on $\Omega$ with bounded total variation, $C_b^0(\Omega)$ the space of all bounded, continuous functions defined on $\Omega$, so that $\int_\Omega{\varphi d\mu}$ is is well defined for $\varphi \in C_b^0(\Omega)$ and $\mu \in \mathfrak{M}_b(\Omega)$.

The positive part, the negative part and total variation of a measure $\mu$ in $\mathfrak{M}_b(\Omega)$ are denoted by $\mu^+, \mu^-$ and $|\mu|$ - is a bounded positive measure on $\Omega$, respectively.

\begin{definition}
\label{def:convergemeas}
A sequence $\{\mu_n\}$ of measures in $\mathfrak{M}_b(\Omega)$ converges to a measure $\mu$ in $\mathfrak{M}_b(\Omega)$ in a narrow topology if:
\begin{align}
\lim_{n \to +\infty}{\int_\Omega{\varphi d \mu_n}} = \int_\Omega{\varphi d\mu},
\end{align}
\label{eq:meas1}
for every $\varphi \in C_b^0(\Omega)$.
\end{definition}

\begin{remark}
\label{rem:meas1}
If $\mu_n$ is nonnegative, then $\{\mu_n\}$ converges to $\mu$ in the narrow topology of measures if and only if $\mu_n(\Omega)$ converges to $\mu(\Omega)$ and \eqref{eq:meas1} holds for every $\varphi \in C_c^\infty(\Omega)$. In particular, if $\mu_n \ge 0$, $\{\mu_n\}$ converges to $\mu$ in the narrow topology of measures if and only if one has \eqref{eq:meas1} holds for every $\varphi \in C^\infty(\overline{\Omega})$.
\end{remark}

In addition, one defines $\mathfrak{M}_0(\Omega)$ as the set of all measures $\mu$ in $\mathfrak{M}_b(\Omega)$ which are ``absolutely continuous'' with respect to the $p$-capacity, i.e., which satisfy $\mu(B)=0$ for every Borel set $B \subseteq \Omega$ such that $\text{cap}_p(B,\Omega)=0$.

\begin{remark}
\label{rem:meas}
For every measure $\mu$ in $\mathfrak{M}_b(\Omega)$ there exists a unique pair of measures $(\mu_0, \mu_s)$, with $\mu_0$ in $\mathfrak{M}_b(\Omega)$ and $\mu_s$ in $\mathfrak{M}_s(\Omega)$, such that $\mu = \mu_0+\mu_s$, is $\mu$ is nonnegative, so are $\mu_0$ and $\mu_s$.
\end{remark}

The measures $\mu_0$ and $\mu_s$ will be called the \emph{absolutely continuous} and the \emph{singular} part of $\mu$ with respect to the $p$-capacity. 

\subsection{Assumptions on operators}
\label{sec:assump_opers}
Let the nonlinearity operator $A: \Omega \times \mathbb{R}^n \rightarrow \mathbb{R}^n$ be a Carath\'eodory function (that is, $A(.,\xi)$ is measurable on $\Omega$ for every $\xi$ in $\mathbb{R}^n$, and $A(x,.)$ is continuous on $\mathbb{R}^n$ for almost every $x$ in $\Omega$) which satisfies the following growth and monotonicity conditions: for some $1<p\le n$:
\begin{align}
\label{eq:A1}
\left| A(x,\xi) \right| &\le \beta |\xi|^{p-1},
\end{align}
\begin{align}
\label{eq:A2}
\langle A(x,\xi)-A(x,\eta), \xi - \eta \rangle &\ge \alpha \left( |\xi|^2 + |\eta|^2 \right)^{\frac{p-2}{2}}|\xi - \eta|^2,
\end{align}
for every $(\xi,\eta) \in \mathbb{R}^n \times \mathbb{R}^n \setminus \{(0,0)\}$ and a.e. $x \in \mathbb{R}^n$, $\alpha$ and $\beta$ are positive constants.

A consequence of \eqref{eq:A1}, and of the continuity of $A$ with respect to $\xi$, is that, for almost every $x$ in $\Omega$,
\begin{align*}
A(x,0) = 0.
\end{align*}

From above hypotheses, the map $u \mapsto -\div(A(x,\nabla u))$ is a coercive, continuous, bounded, and monotone operator defined on $W^{1,p}(\Omega)$ with values in its dual space $W^{-1,p'}(\Omega)$. Moreover, by the theory of monotone operators, for every $\mu$ in $W^{-1,p'}(\Omega)$ there exists one and only one solution $v$ of the problem
\begin{align*}
\begin{cases}
-\div(A(x,\nabla v)) &= \mu, \quad \text{on} \ \Omega, \\
v &= 0, \quad \text{on} \ \partial\Omega,
\end{cases}
\end{align*}
in the sense that:
\begin{align*}
\begin{cases}
v \in W^{1,p}_0(\Omega), \\
\displaystyle{\int_\Omega{A(x,\nabla v) \cdot \nabla \varphi dx} }= \langle \mu,\varphi \rangle, \quad \forall \varphi \in W^{1,p}_0(\Omega),
\end{cases}
\end{align*}
where $\langle .,. \rangle$ denotes the duality between $W^{-1,p'}(\Omega)$ and $W^{1,p}_0(\Omega)$. If $p>n$, the $\mathfrak{M}_b(\Omega)$ is a subset of $W^{-1,p'}(\Omega)$, so that this classical result gives the existence and uniqueness of a solution to \eqref{eq:elliptictype} for every measure $\mu$ in $\mathfrak{M}_b(\Omega)$.

\subsection{Definition of renormalized solution}

For each integer $k>0$, and for $s \in \mathbb{R}$ we firstly define the operator $T_k: \mathbb{R} \to \mathbb{R}$ as:
\begin{align}
\label{eq:Tk}
T_k (s) = \max\left\{ -k,\min\{k,s\} \right\},
\end{align}
and this belongs to $W_0^{1,p}(\Omega)$ for every $k>0$, which satisfies
\begin{align*}
-\div A(x,\nabla T_k(u)) = \mu_k
\end{align*}
in the sense of distribution in $\Omega$ for a finite measure $\mu_k$ in $\Omega$. 
\begin{definition}
\label{def:truncature}
Let $u$ be a measurable function defined on $\Omega$ which is finite almost everywhere, and satisfies $T_k(u) \in W^{1,1}_0(\Omega)$ for every $k>0$. Then, there exists a unique measurable function $v: \Omega \to \mathbb{R}^n$ such that:
\begin{align}
\nabla T_k(u) = \chi_{\{|u| \le k\}} v , \quad \text{almost everywhere in} \ \ \Omega, \quad \text{for  every} \ k>0.
\end{align}
Moreover, the function $v$ is so-called `` distributional gradient $\nabla u$'' of $u$.
\end{definition}

Let us recall the Remark \ref{rem:meas}, for every measure $\mu$ in $\mathfrak{M}_b(\Omega)$ can be written in a unique way as $\mu = \mu_0+\mu_s$, where $\mu_0$ in $\mathfrak{M}_0(\Omega)$ and $\mu_s$ in $\mathfrak{M}_s(\Omega)$.

The following Definition \ref{def:renormsol3} of renormalized solution to equation \eqref{eq:elliptictype} was introduced in \cite{11DMOP}, and we reproduce them herein as:

\begin{definition}
\label{def:renormsol3}
Let $\mu = \mu_0+\mu_s \in \mathfrak{M}_b(\Omega)$, where $\mu_0 \in \mathfrak{M}_0(\Omega)$ and $\mu_s \in \mathfrak{M}_s(\Omega)$. A measurable function $u$ defined in $\Omega$ and finite almost everywhere is called a renormalized solution of \eqref{eq:elliptictype} if $T_k(u) \in W^{1,p}_0(\Omega)$ for any $k>0$, $|{\nabla u}|^{p-1}\in L^r(\Omega)$ for any $0<r<\frac{n}{n-1}$, and $u$ has the following additional property. For any $k>0$ there exist  nonnegative Radon measures $\lambda_k^+, \lambda_k^- \in\mathfrak{M}_0(\Omega)$ concentrated on the sets $u=k$ and $u=-k$, respectively, such that $\mu_k^+\rightarrow\mu_s^+$, $\mu_k^-\rightarrow\mu_s^-$ in the narrow topology of measures and  that
 \begin{align*}
 \int_{\{|u|<k\}}\langle A(x,\nabla u),\nabla \varphi\rangle
  	dx=\int_{\{|u|<k\}}{\varphi d}{\mu_{0}}+\int_{\Omega}\varphi d\lambda_{k}%
  	^{+}-\int_{\Omega}\varphi d\lambda_{k}^{-},
 \end{align*}
  	for every $\varphi\in W^{1,p}_0(\Omega)\cap L^{\infty}(\Omega)$.
\end{definition}


It is known that if $\mu\in \mathfrak{M}_0(\Omega)$ then there is one and only one  renormalized solution of \eqref{eq:elliptictype} (see \cite{11DMOP}). However, for a general $\mu\in \mathfrak{M}_b(\Omega)$ the uniqueness of renormalized solutions of \eqref{eq:elliptictype} is still an open problem. 

The following Remark was given in \cite[Theorem 4.1]{11DMOP} provides the gradient estimate for solution $u$:
\begin{remark}
	\label{rem:nablau}
	Let $\Omega$ is an open bounded domain in $\mathbb{R}^n$. Then, there exists $C=C(n,p)$ such that for any the renormalized solution $u$ to \eqref{eq:elliptictype} with a given finite measure data $\mu$ there holds:
	\begin{align}
	\label{eq:nablau}
	\|\nabla u\|_{L^{\frac{(p-1)n}{n-1},\infty}(\Omega)}\leq C\left[|\mu|(\Omega)\right]^{\frac{1}{p-1}}.
	\end{align}
\end{remark}

\begin{proposition}
	\label{prop:nablauconverge}
Let $\mu \in L^1(\Omega)$ and a sequence $(u_k)_k$ be the renormalized solution of \eqref{eq:elliptictype} with data $\mu_k \in L^{\frac{p}{p-1}}(\Omega)$ such that $\mu_k \to \mu$ weakly in $L^1(\Omega)$. Then, there exists a subsequence $\{u_{k'}\}_{k'}$ which converges to $u$ in $L^s(\Omega)$ the renormalized solution to \eqref{eq:elliptictype} with measure data $\mu$, for all $s < \frac{(p-1)n}{n-p}$. Moreover, $\nabla u_{k'} \to \nabla u$ in $L^q(\Omega)$ for all $q<\frac{(p-1)n}{n-1}$.
\end{proposition}

\subsection{Other definitions and remarks}
\label{sec:othersdefs}
Let us recall the definition of the Lorentz space $L^{q,t}(\Omega)$ for $0<q<\infty$ and $0<t\le \infty$ (see in \cite{55Gra}). It is the set of all Lebesgue measurable functions $g$ on $\Omega$ such that:
\begin{align}
\label{eq:lorentz}
\|g\|_{L^{q,t}(\Omega)} = \left[ q \int_0^\infty{\left( \lambda^q| \{x \in \Omega: |g(x)|>\lambda\} \right)^{\frac{t}{q}} \frac{d\lambda}{\lambda}} \right]^{\frac{1}{t}} < +\infty,
\end{align}
as $t \neq \infty$. If $t = \infty$, the space $L^{q,t}(\Omega)$ is the usual weak $L^q$ or Marcinkiewicz space with the following quasinorm:
\begin{align}
\|g\|_{L^{q,\infty}(\Omega)} = \sup_{\lambda>0}{\lambda \left|\{x \in \Omega:|g(x)|>\lambda\}\right|^{\frac{1}{q}}},
\end{align}
where $|B|$ denotes the $n$-dimensional Lebesgue measure of a set $B \subset \mathbb{R}^n$. In \eqref{eq:lorentz}, for $t=q$, the Lorentz space $L^{q,q}(\Omega)$ is the Lebesgue space $L^q(\Omega)$. In addition, let us recall that for $1<r<q<\infty$, one has:
\begin{align}
L^q(\Omega) \subset L^{q,\infty}(\Omega) \subset L^r(\Omega).
\end{align}

In this paper, we also define the  the fractional  maximal function $\mathbf{M}_\alpha$ of each locally finite measure $\mu$ by:
\begin{align}
\label{eq:Malpha}
\mathbf{M}_\alpha(\mu)(x) = \sup_{\rho>0}{\frac{|\mu|(B_\rho(x))}{\rho^{n-\alpha}}}, \quad \forall x \in 
\mathbb{R}^n, ~0<\alpha<n.
\end{align}

For the case $\alpha=0$, the definition of $\mathbf{M}_\sigma$ becomes $\mathbf{M}_0$ is essentially the Hardy-Littlewood maximal function $\mathbf{M}$ defined for each locally integrable function $f$ in $\mathbb{R}^n$ by:
\begin{align}
\label{eq:M0}
\mathbf{M}(f)(x) = \sup_{\rho>0}{\fint_{B_{\rho}(x)}|f(y)|dy},~~ \forall x \in \mathbb{R}^n,
\end{align}
where the denotation $\displaystyle{\fint_{B_r(x)}{f(y)dy}}$ indicates the integral average of $f$ in the variable $y$ over the ball $B_r(x)$, i.e.
\begin{align*}
\fint_{B_r(x)}{f(y)dy} = \frac{1}{|B_\rho(x)|}\int_{B_r(x)}{f(y)dy}.
\end{align*}
\begin{remark}
\label{rem:boundM}
It refers to \cite{55Gra} that the operator $\mathbf{M}$ is bounded from $L^s(\mathbb{R}^n)$ to $L^{s,\infty}(\mathbb{R}^n)$, for $s \ge 1$, this means,
\begin{align}
|\mathbf{M}(g)>\lambda| \le \frac{C}{\lambda^s}\int_{\mathbb{R}^n}{|g|^s dx}.
\end{align}
\end{remark}

Our result the good-$\lambda$ type inequality and gradient estimate in Theorem \ref{theo:lambda_estimate} and \ref{theolorentz_estimate} will be proved in Section \ref{sec:main} involving both these above operators $\mathbf{M}_1$ and $\mathbf{M}$.

\section{Local interior and boundary comparison estimates}
\label{sec:lems}

In this section, we obtain certain local interior and boundary comparison
estimates that are essential to our development later. First, let us consider the interior ones. Fix a point $x_0 \in \Omega$, for $0<2R \le r_0$ ($r_0$ was given in \eqref{eq:capuni}) and $\mu \in \mathfrak{M}_b(\Omega)$, with $u\in W_{loc}^{1,p}(\Omega)$ being solution to \eqref{eq:elliptictype} and for each ball $B_{2R}=B_{2R}(x_0)\subset\subset\Omega$, we consider the unique solution $w\in W_{0}^{1,p}(B_{2R})+u$ to the  equation:
\begin{equation}
\label{111120146}\left\{ \begin{array}{rcl}
- \operatorname{div}\left( {A(x,\nabla w)} \right) &=& 0 \quad \text{in} \quad B_{2R}, \\ 
w &=& u\quad \text{on} \quad \partial B_{2R}.  
\end{array} \right.
\end{equation}

We first recall the following  version of interior Gehring's lemma applied to the function $w$ defined in \eqref{111120146}, that was actually part of \cite[Lemma 3.3]{Mi2}.

\begin{lemma} \label{111120147} 
Let $u \in W^{1,p}_{\text{loc}}(\Omega)$ and $w$ be the solution to \eqref{111120146}. Then, there exist  constants $\Theta = \Theta(n,p,\alpha, \beta)>p$ and $C = C(n,p,\alpha,\beta)>0$ such that the following estimate      
	\begin{equation}\label{111120148}
	\left(\fint_{B_{\rho/2}(y)}|\nabla w|^{\Theta} dxdt\right)^{\frac{1}{\Theta}}\leq C\left(\fint_{B_{\rho}(y)}|\nabla w|^{p-1} dx\right)^{\frac{1}{p-1}}
	\end{equation}
	holds for all  $B_{\rho}(y)\subset B_{2R}(x_0)$. 
\end{lemma} 

The next lemma gives an estimate for the difference $\nabla(u-w)$, with $p$ satisfies $\frac{3n-2}{2n-1}< p\leq 2-\frac{1}{n}$. These results were proved by Q.H. Nguyen in \cite[Lemma 2.2, 2.3]{55QH4}.

\begin{lemma}
\label{lem:estimateinter}
Let $u \in W^{1,p}_{\text{loc}}(\Omega)$ and $w$ be solution to \eqref{111120146} and assume that $\frac{3n-2}{2n-1}<p\leq 2-\frac{1}{n}$. Then, there is a constant  $C = C(n,p,\alpha,\beta)>0$ such that:
	\begin{align}\nonumber
	\left(	\fint_{B_{2R}(x_0)}|\nabla (u-w)|^{\gamma_0}dx\right)^{\frac{1}{\gamma_0}}&\leq C\left[\frac{|\mu|(B_{2R}(x_0))}{R^{n-1}} \right]^{\frac{1}{p-1}}+\\
	&\qquad +	C \frac{|\mu|(B_{2R}(x_0))}{R^{n-1}} \left(	\fint_{B_{2R}(x_0)}|\nabla u|^{\gamma_0}dx\right)^{\frac{2-p}{\gamma_0}},	\label{eq:estimateinter}
	\end{align}
	for some $\frac{2-p}{2}\leq \gamma_0<\frac{(p-1)n}{n-1}\leq 1$.
\end{lemma}


In addition, it remarks that throughout this paper, we have $\gamma_0>p-1$. Next, let us also recall the counterparts of Lemmas \ref{111120147} and \ref{lem:estimateinter} up to the boundary. As $\mathbb{R}^n \setminus \Omega$ is uniformly $p$-thick with constants $c_0, r_0>0$, let $x_0 \in \partial \Omega$ be a boundary point and for $0<R<r_0/10$ we set $\Omega_{10R} = \Omega_{10R}(x_0) = B_{2R}(x_0) \cap \Omega$. For $u \in W^{1,p}_0(\Omega)$ being a solution to \eqref{eq:elliptictype}, we consider the unique solution $w \in u+W^{1,p}_0(\Omega_{10R})$ to the following equation:
\begin{equation}
\label{111120146*}
\left\{ \begin{array}{rcl}
- \operatorname{div}\left( {A(x,\nabla w)} \right) &=& 0 \quad ~~~\text{in}\quad \Omega_{10R}(x_0), \\ 
w &=& u\quad \quad \text{on} \quad \partial \Omega_{10R}(x_0). 
\end{array} \right.
\end{equation}

In what follows we extend $\mu$ and $u$ by zero to $\mathbb{R}^n \setminus \Omega$ and $w$ by $u$ to $\mathbb{R}^n \setminus \Omega_{10R}$. And let us recall the following lemma \ref{111120147*}, which was stated and proved in \cite{55Ph0}.

\begin{lemma} \label{111120147*} 
Let $w$ be as in \eqref{111120146*} and $c_0$ is the constant in Definition \ref{def:pcapuni}. Then, there exist  constants $\Theta=\Theta(n,p,\alpha, \beta,c_0)>p$ and $C = C(n,p,\alpha,\beta,c_0)>0$ such that the following estimate    
	\begin{equation}\label{111120148*}
	\left(\fint_{B_{\rho/2}(y)}|\nabla w|^{\Theta} dx\right)^{\frac{1}{\Theta}}\leq C\left(\fint_{B_{3\rho}(y)}|\nabla w|^{p-1} dx\right)^{\frac{1}{p-1}}
	\end{equation} 
holds for all  $B_{3\rho}(y) \subset B_{10R}(x_0)$, $y \in B_r(x_0)$. 
\end{lemma}

\begin{lemma} 
\label{111120147**} 
Let $w$ be as in \eqref{111120146*} and $c_0$ is the constant in Definition \ref{def:pcapuni}. Then, there exist  constants $\Theta=\Theta(n,p,\alpha, \beta,c_0)>p$ and $C = C(n,p,\alpha,\beta,c_0)>0$ such that we have the following estimate      
	\begin{equation}\label{111120148**}
	\left(\fint_{B_{\rho/2}(y)}|\nabla w|^{\Theta} dx\right)^{\frac{1}{\Theta}}\leq C\left(\fint_{B_{2\rho/3}(y)}|\nabla w|^{p-1} dx\right)^{\frac{1}{p-1}}
	\end{equation} 
	holds for all  $B_{\rho}(y)\subset B_{10R}(x_0)$, $y \in B_r(x_0)$. 
\end{lemma}
\begin{proof}
Note that there exist $m=m(d)$ and $x_1,...,x_{m}\in B_{1/2}(0)$ such that 
\begin{align*}
B_{1/2}(0)\subset B_{\frac{1}{1000}}(x_1)\cup...\cup B_{\frac{1}{1000}}(x_m)
\end{align*}

For $\rho'>0$, let $B_{\rho'}(y) \subset B_{10R}(x_0)$ and it is easy to check that
\begin{align}
\label{eq:lem351}
B_{\rho'/2}(y)\subset B_{\frac{\rho'}{1000}}(y+\rho'x_1)\cup...\cup B_{\frac{\rho'}{1000}}(y+\rho'x_m).
\end{align}

Since $B_{\frac{6\rho'}{1000}}(y+\rho'x_i)\subset B_{\frac{2\rho'}{3}}(y), \forall i=1,2,...,m$, one can apply \eqref{111120148*}, there exist $\Theta=\Theta(n,p,\alpha, \beta,c_0)$ and $C = C(n,p,\alpha,\beta,c_0)>0$ such that
\begin{align}
\label{eq:lem352}
\left(\fint_{B_{\frac{\rho'}{1000}}(y+\rho'x_i)}|\nabla w|^{\Theta} dxdt\right)^{\frac{1}{\Theta}}\leq C\left(\fint_{B_{\frac{6\rho'}{1000}}(y+\rho'x_i)}|\nabla w|^{p-1} dx\right)^{\frac{1}{p-1}}~\forall~i=1,...,m.
\end{align}

Thus, from \eqref{eq:lem351} and \eqref{eq:lem352} one obtains:
\begin{align*}
\left(\fint_{B_{\frac{\rho'}{2}}(y)}|\nabla w|^{\Theta} dxdt\right)^{\frac{1}{\Theta}}&\leq C\sum_{i=1}^{m} \left(\fint_{B_{\frac{\rho'}{1000}}(y+\rho'x_i)}|\nabla w|^{\Theta} dxdt\right)^{\frac{1}{\Theta}}\\& \leq  C\sum_{i=1}^{m} \left(\fint_{B_{\frac{6\rho'}{1000}}(y+\rho'x_i)}|\nabla w|^{p-1} dxdt\right)^{\frac{1}{p-1}}
\\& \leq  C \left(\fint_{B_{\frac{2\rho'}{3}}(y)}|\nabla w|^{p-1} dxdt\right)^{\frac{1}{p-1}},
 \end{align*}
and the proof is complete.
\end{proof}\\\\
More general, we also obtain the following lemma.
\begin{lemma}
\label{lem:addon12}
Let $w$ be as in \eqref{111120146*} and $c_0$ is the constant in Definition \ref{def:pcapuni}. Then, for $0<\theta_1<\theta_2<1$ there exist  constants $\Theta=\Theta(n,p,\alpha, \beta,c_0)>p$ and $C = C(n,p,\theta_1, \theta_2,\alpha,\beta,c_0)>0$ such that we have the following estimate      
	\begin{equation}\label{111120148**}
	\left(\fint_{B_{\theta_1\rho}(y)}|\nabla w|^{\Theta} dx\right)^{\frac{1}{\Theta}}\leq C\left(\fint_{B_{\theta_2\rho}(y)}|\nabla w|^{p-1} dx\right)^{\frac{1}{p-1}}
	\end{equation} 
	holds for all  $B_\rho(y) \subset B_{10R}(x_0)$, $y \in B_r(x_0)$. 
\end{lemma}

\begin{lemma}
\label{lem:estimatebound}
Let $u \in W^{1,p}_{\text{loc}}(\Omega)$ and $w$ be solution to \eqref{111120146*}. Assume that $\frac{3n-2}{2n-1}<p\leq 2-\frac{1}{n}$. Then, there is a constant $C = C(n,p,\alpha,\beta,c_0)>0$ such that:
	\begin{align}
	\begin{split}
	\left(	\fint_{B_{10R}(x_0)}|\nabla (u-w)|^{\gamma_0}dx\right)^{\frac{1}{\gamma_0}}&\leq C\left[ \frac{|\mu|(B_{10R}(x_0))}{R^{n-1}}  \right]^{\frac{1}{p-1}}+\\
	&\qquad +	C\frac{|\mu|(B_{10R}(x_0))}{R^{n-1}} \left(	\fint_{B_{10R}(x_0)}|\nabla u|^{\gamma_0}dx\right)^{\frac{2-p}{\gamma_0}},	\label{eq:estimateinter'}
	\end{split}
	\end{align}
	for some $\frac{2-p}{2}\leq \gamma_0<\frac{(p-1)n}{n-1}\leq 1$.
	\end{lemma}

\section{Good-$\lambda$ type bound and Gradient Lorentz Estimate}
\label{sec:main}
In this section, we state the main result of this paper. Let us give a proof of Theorem \ref{theo:lambda_estimate} on solution to \eqref{eq:elliptictype}. Let us recall here that the domain $\Omega$ is assumed to satisfy the $p$-capacity uniform thickness condition with constants $c_0>0, r_0>0$ as in Definition \ref{def:pcapuni}. The following Lemma is important and used to prove the main Theorem, it can be viewed as a substitution for the Calder\'on-Zygmund-Krylov-Safonov decomposition. 
\begin{lemma}
\label{lem:mainlem}
Let $0<\varepsilon<1, R \ge R_1>0$ and the ball $Q:=B_R(x_0)$ for some $x_0\in \mathbb{R}^n$.  Let $E\subset F\subset Q$ be two measurable sets in $\mathbb{R}^{n+1}$ with $|E|<\varepsilon |B_{R_1}|$ and satisfying the following property: for all $x\in Q$ and $r\in (0,R_1]$, we have $B_r(x)\cap Q\subset F$  	provided $|E\cap B_r(x)|\geq \varepsilon |B_r(x)|$.	Then $|E|\leq B\varepsilon |F|$ for some $B=B(n)$.
\end{lemma}
\begin{proof}[Proof of Theorem \ref{theo:lambda_estimate}]
Let $\gamma_0$ satisfy $\frac{2-p}{2} \le \gamma_0 < \frac{(p-1)n}{n-1}<1$ given as in Lemmas \ref{lem:estimateinter} and \ref{lem:estimatebound}. 

Let $u$ be the renormalized solution to \eqref{eq:elliptictype}. From Remark \ref{rem:nablau}, we have \eqref{eq:nablau} which implies that
\begin{align}
\label{es14}
\left(	\frac{1}{T_0^n}\int_{\Omega}|\nabla u|^{\gamma}\right)^{1/\gamma}\leq C_\gamma \left[\frac{|\mu|(\Omega)}{T_0^{n-1}}\right]^{\frac{1}{p-1}}, \qquad \text{with}\quad T_0=diam(\Omega),
\end{align}
for any $\gamma\in \left(0,\frac{(p-1)n}{n-1}\right)$.

Let $\mu_0,\lambda_k^+,\lambda_k^-$ be as in Definition \ref{def:renormsol3}. Let $u_k\in W_0^{1,p}(\Omega)$ be the unique solution to the following problem:
	\begin{equation*}
	\left\{
	\begin{array}[c]{rcl}
	-\text{div}(A(x,\nabla u_k))&=&\mu_{k} \quad \text{in } \quad\Omega,\\
	{u}_{k}&=&0\quad \ \text{on } \quad \partial\Omega,\\
	\end{array}
	\right.  
	\end{equation*}
	where $\mu_k=\chi_{|u|<k}\mu_0+\lambda_k^+-\lambda_k^-$.

For given $\varepsilon>0, \lambda>0$ and $r_0>0$, let us set 
\begin{align*}
E_{\lambda,\varepsilon}=\{({\bf M}(|\nabla u|^{\gamma_0}))^{1/\gamma_0}>\varepsilon^{-\frac{1}{\Theta}}\lambda, (\mathbf{M}_1(\mu))^{\frac{1}{p-1}}\le \varepsilon^{\frac{1}{(p-1)\gamma_0}}\lambda \}\cap Q,
\end{align*}
and 
\begin{align*}
F_\lambda=\{ ({\bf M}(|\nabla u|^{\gamma_0}))^{1/\gamma_0}> \lambda\}\cap Q,
\end{align*}
for $\lambda>0$. Here one remarks that  $Q=B_{2T_0}(x_0)$, for $x_0\in \Omega$, $T_0=diam(\Omega)$. 

Our purpose here is to prove that there exist $\Theta$ and $C$ such that \eqref{eq:mainlambda} holds. It can be rewritten as:
$$|E_{\lambda,\varepsilon}| \le C\varepsilon |F_\lambda|,$$ 
and therein Lemma \ref{lem:mainlem} has to be used, in which one needs to prove firstly that:
	 \begin{equation}\label{5hh2310131}
	 |E_{\lambda,\varepsilon}|\leq C\varepsilon |{B_{R_0}}(0)|~~\forall \lambda>0, 
	 \end{equation}
where $R_0=\min\{T_0,r_0\}$.  Indeed, we may assume that $E_{\lambda,\varepsilon}\not=\emptyset$ (if $E_{\lambda,\varepsilon}=\emptyset$, \eqref{5hh2310131} holds obiviously). Then, there is $x_1 \in Q$ such that
\begin{align*}
\left(\mathbf{M}_1(\mu)(x_1) \right)^{\frac{1}{p-1}} \le \varepsilon^{\frac{1}{(p-1)\gamma_0}}\lambda
\end{align*}
which implies
	\begin{align}
	\label{eq:bt3}
	|\mu| (\Omega)\leq T_0^{n-1}(\varepsilon\lambda)^{p-1}.
	\end{align}
	Thanks to Remark \ref{rem:boundM}, with $s=1$, $g = (\nabla u)^{\gamma_0}$ and $t=\left(\varepsilon^{-\frac{1}{\Theta}}\lambda\right)^{\gamma_0}$, in the view of \eqref{es14} with $\gamma=\gamma_0$ one has:
	 \begin{align}
	 \label{eq:bt5}
	  |E_{\lambda,\varepsilon}|&\leq \frac{C}{(\varepsilon^{-\frac{1}{\Theta}}\lambda)^{\gamma_0}}\int_{\Omega}|\nabla u|^{\gamma_0}dx \leq \frac{CT_0^n}{(\varepsilon^{-\frac{1}{\Theta}}\lambda)^{\gamma_0}} \left[\frac{|\mu|(\Omega)}{T_0^{n-1}}\right]^{\frac{\gamma_0}{p-1}}.
	 \end{align}
 In the use of \eqref{eq:bt3} we get that
\begin{align}
\begin{split}
|E_{\lambda,\varepsilon}| \leq \frac{CT_0^n}{(\varepsilon^{-\frac{1}{\Theta}}\lambda)^{\gamma_0}} \left[\frac{T_0^{n-1}(\varepsilon^{\frac{1}{(p-1)\gamma_0}}\lambda)^{p-1}}{T_0^{n-1}}\right]^{\frac{\gamma_0}{p-1}} &= CT_0^n\varepsilon^{\left(\frac{1}{\Theta}+\frac{1}{(p-1)\gamma_0}\right)\gamma_0} \\ &=C\varepsilon^{(\frac{1}{\Theta}+\frac{1}{(p-1)\gamma_0})\gamma_0}|B_{R_0}| \\ &= \frac{C}{|B_1|}\left(\frac{T_0}{R_0} \right)^n \varepsilon |B_{R_0}| 
\end{split}
\end{align}
which implies $|E_{\lambda,\varepsilon}| \le C\varepsilon|B_{R_0}|$, in which $C$ depending on $(T_0/R_0)^n$ and so, \eqref{5hh2310131} is well proved.


	 Next we verify that for all $x\in Q$, $r\in (0,2R_0]$, and $\lambda>0$ we have:
	 \begin{equation}\label{eq:bt6}
	|E_{\lambda,\varepsilon}\cap B_r(x)|\geq C\varepsilon |B_r(x)| \Longrightarrow B_r(x)\cap Q\subset F_\lambda.
	 \end{equation}
	 	
	Indeed, let $x\in Q$ and $0<r\leq 2R_0$, and by contradiction, let us assume that $B_r(x)\cap \Omega\cap F^c_\lambda\not= \emptyset$ and $E_{\lambda,\varepsilon}\cap B_r(x)\not = \emptyset$. Then, there exist $x_1,x_2\in B_r(x)\cap Q$ such that
	 \begin{align}
	 \label{eq:bt7}
	 \left[{\bf M}(|\nabla u|^{\gamma_0})(x_1)\right]^{1/\gamma_0}\leq \lambda,
\end{align}	 	 
	 and 
	 \begin{align}
	 \label{eq:bt8}
	 \mathbf{M}_1(\mu)(x_2)\le (\varepsilon^{ \frac{1}{(p-1)\gamma_0}} \lambda)^{p-1}.
	 \end{align}
	 One needs to prove that there exists a constant $C$ depending on $n,p,\alpha, \beta, \sigma,\gamma_0,c_0$ such that the following estimate holds:
	 \begin{equation}\label{5hh2310133}
	|E_{\lambda,\varepsilon}\cap B_r(x)|< C \varepsilon |B_r(x)|. 
	 \end{equation}
	 
	For $\rho>0$, firstly we have
	\begin{align*}
	\left(\fint_{B_\rho(y)}|\nabla u|^{\gamma_0}\right)^{1/\gamma_0}\leq  \sup\left\{\left(\sup_{\rho'<r}			\fint_{B_{\rho'}(y)}\chi_{B_{2r}(x)}|\nabla u|^{\gamma_0}\right)^{1/\gamma_0}, \left(\sup_{\rho'\geq r}\fint_{B_{\rho'}(y)}|\nabla u|^{\gamma_0}\right)^{1/\gamma_0}\right\}.
	\end{align*}
	
	For  $y\in B_r(x),\rho'\geq r$, one has $B_{\rho'}(y)\subset B_{\rho'+r}(x)\subset B_{\rho'+2r}(x_1)\subset B_{3\rho'}(x_1)$. Thus, by \eqref{eq:bt7}:
	\begin{align*}
	\left(\fint_{B_\rho(y)}|\nabla u|^{\gamma_0}\right)^{1/\gamma_0}&\leq  \sup\left\{\left[{\bf M}\left(\chi_{B_{2r}(x)}|\nabla u|^{\gamma_0}\right)(y)\right]^{\frac{1}{\gamma_0}}, 3^{\frac{n}{\gamma_0}} \left(\sup_{\rho'>0}\fint_{B_{\rho'}(x_1)}|\nabla u|^{\gamma_0}\right)^{1/\gamma_0}\right\}\\& \leq  \sup\left\{\left[{\bf M}\left(\chi_{B_{2r}(x)}|\nabla u|^{\gamma_0}\right)(y)\right]^{\frac{1}{\gamma_0}}, 3^{\frac{n}{\gamma_0}} \lambda\right\},
	\end{align*}
therein \eqref{eq:bt7} has been used for the last inequality.

	Take to $\text{sup}$ both sides for $\rho$, it can be seen clearly that:
	 \begin{equation*}
	\left( {\bf M}(|\nabla u|^{\gamma_0})(y)\right)^{1/\gamma_0}\leq \max\{\left[{\bf M}\left(\chi_{B_{2r}(x)}|\nabla u|^{\gamma_0}\right)(y)\right]^{\frac{1}{\gamma_0}},3^{\frac{n}{\gamma_0}}\lambda\},~\forall y\in B_r(x).
	 \end{equation*}
	 Therefore, for all $\lambda>0$ and $\varepsilon_0$ satisfies $\varepsilon_0^{-\frac{1}{\Theta}}>3^{\frac{n}{\gamma_0}}$, we will get that
	 \begin{align}
	 \label{5hh2310134}
	 \begin{split}
	 E_{\lambda,\varepsilon}\cap B_r(x) =\{{\bf M}\left(\chi_{B_{2r}(x)}|\nabla u|^{\gamma_0}\right)^{\frac{1}{\gamma_0}}>\varepsilon^{-\frac{1}{\Theta}}\lambda, (\mathbf{M}_{1}(\mu))^{\frac{1}{p-1}}\leq\varepsilon^{\frac{1}{(p-1)\gamma_0}}\lambda\} \\ ~~~~~\cap Q \cap B_r(x)
	 \end{split}
	 \end{align}
	 for all $\varepsilon \in (0,\varepsilon_0)$.
	 
	  In order to prove \eqref{5hh2310133} we separately consider for the case $B_{4r}(x)\subset\subset \Omega$ and the case $B_{4r}(x)\cap \Omega^{c}\not=\emptyset$.\\ \medskip\\
	   \textbf{Case 1: $B_{4r}(x)\subset\subset\Omega$}: Applying Lemma \ref{lem:estimateinter} for $u_k \in W^{1,p}_{loc}(\Omega)$ and $w_k$ the solution to:
	\begin{equation}
	\label{eq:wsol}
	\begin{cases}
	\div(A(x,\nabla w_k) &=0, \quad \text{in}\ \ B_{4r}(x)\\
	w_k &= u_k, \quad \text{on} \ \ \partial B_{4r}(x),
	\end{cases}
	\end{equation}
with $\mu = \mu_k$ and $B_{2R} = B_{4r}(x)$, one has a constant $C = C(n,p,\alpha,\beta)>0$ such that:
	\begin{align}
	\label{eq:btgeneral}
	\begin{split}
	\left(\fint_{B_{4r}(x)}{|\nabla u_k - \nabla w_k|^{\gamma_0}dx} \right)^{\frac{1}{\gamma_0}} &\le C\left[\frac{|\mu_k|(B_{4r}(x))}{r^{n-1}}\right]^{\frac{1}{p-1}} \\ &+ C\frac{|\mu_k|(B_{4r}(x))}{r^{n-1}} \left(\fint_{B_{4r}(x)}{|\nabla u_k|^{\gamma_0}dx} \right)^{\frac{2-p}{\gamma_0}}.
	\end{split}
	\end{align}
	Otherwise, Lemma \ref{111120147} is also applied to give:
	\begin{align}
\label{eq:bt17}
\begin{split}
	\left(\fint_{B_{2r}(x)}|\nabla w_k|^\Theta dx\right)^{\frac{1}{\Theta}}  &\leq C \left(\fint_{B_{4r}(x)}|\nabla w_k|^{p-1} dx\right)^{\frac{1}{p-1}}, \\ 
	&\le C \left(\fint_{B_{4r}(x)}{|\nabla u_k|^{\gamma_0}dx} \right)^{\frac{1}{\gamma_0}}+ \\&+ C\left( \fint_{B_{4r}(x)}{|\nabla u_k - \nabla w_k|^{\gamma_0}dx} \right)^{\frac{1}{\gamma_0}},
\end{split}
\end{align}
where, the second inequality is obtained by using Holder's inequality and for $\gamma_0>p-1$.

	For all $m \ge 2$ and $\gamma_0<1$, since one has that
	\begin{align}
	\label{eq:ineqM}
	\begin{split}
\left[	\mathbf{M}\left(\left|\sum_{i=1}^{m}f_i \right|^{\gamma_0}\right)\right]^{\frac{1}{\gamma_0}} &\leq \left[ \sum_{i=1}^{m}	\mathbf{M}\left(\left|f_i \right|^{\gamma_0}\right)\right]^{\frac{1}{\gamma_0}}\\ &\leq m^{\frac{1}{\gamma_0}-1}\sum_{i=1}^{m}\left[ 	\mathbf{M}(|f_i|^{\gamma_0})\right]^{\frac{1}{\gamma_0}},~\forall f_i \in L^{\gamma_0}(\Omega),
\end{split}
	\end{align}
	and so, apply for $m=3$ yields that:
	\begin{align}
	\label{eq:bt11}
	\begin{split}	
	|E_{\lambda,\varepsilon}\cap B_r(x)|&\leq   |\{{\bf M}\left(\chi_{B_{2r}(x)}|\nabla (u_k-w_k)|^{\gamma_0}\right)^{\frac{1}{\gamma_0}}>3^{-\frac{1}{\gamma_0}}\varepsilon^{- \frac{1}{\Theta}}\lambda\}\cap B_r(x)|+
	  \\&+ |\{{\bf M}\left(\chi_{B_{2r}(x)}|\nabla (u-u_k)|^{\gamma_0}\right)^{\frac{1}{\gamma_0}}>3^{-\frac{1}{\gamma_0}}\varepsilon^{- \frac{1}{\Theta}}\lambda\}\cap B_r(x)|+ \\&+
	  |\{{\bf M}\left(\chi_{B_{2r}(x)}|\nabla w_k|^{\gamma_0}\right)^{\frac{1}{\gamma_0}}>3^{-\frac{1}{\gamma_0}}\varepsilon^{- \frac{1}{\Theta}}\lambda\}\cap B_r(x)|.  
	  \end{split}
	\end{align}

	From Remark \ref{rem:boundM}, for each term on right hand side of \eqref{eq:bt11} one gives
	\begin{align}
	\label{eq:bt12}
	\begin{split}
	|E_{\lambda,\varepsilon}\cap B_r(x)| &\le \frac{C}{(\varepsilon^{-\frac{1}{\Theta}}\lambda)^{\gamma_0}}\left[ \int_{B_{2r}(x)}{|\nabla u_k - \nabla w_k|^{\gamma_0}dx}+ \right. \\& \left. ~~~~~~~~~~~~~~~~~~~+ \int_{B_{2r}(x)}{|\nabla u - \nabla u_k|^{\gamma_0}dx}\right]+ \\ &+\frac{C}{(\varepsilon^{-\frac{1}{\Theta}}\lambda)^{\Theta}}\int_{B_{2r}(x)}{|\nabla w_k|^{\Theta}dx}.
	\end{split}
	\end{align}

	Combining both estimates \eqref{eq:btgeneral} and \eqref{eq:bt17} to \eqref{eq:bt12} we get
	\begin{align*}
	&|E_{\lambda,\varepsilon} \cap B_r(x)| \\&\le C\varepsilon^{\gamma_0\frac{1}{\Theta}}\lambda^{-\gamma_0}r^n \left[\left( \frac{|\mu_k|(B_{4r}(x))}{r^{n-1}}\right)^{\frac{1}{p-1}}+ \right.\\&\left.~~~~~~~~~~~~~~+\frac{|\mu_k|(B_{4r}(x))}{r^{n-1}}\left(\fint_{B_{4r}(x)}{|\nabla u_k|^{\gamma_0}dx} \right)^{\frac{2-p}{\gamma_0}} \right]^{\gamma_0} +\\ &+ C\varepsilon^{\gamma_0\frac{1}{\Theta}}\lambda^{-\gamma_0}\int_{B_{4r}(x)}{|\nabla u - \nabla u_k|^{\gamma_0}dx} +\\
&+ C\varepsilon \lambda^{-\Theta}r^n\left[\left(\int_{B_{4r}(x)}{|\nabla u_k|^{\gamma_0}dx} \right)^{\frac{1}{\gamma_0}} +  \left( \frac{|\mu_k|(B_{4r}(x))}{r^{n-1}}\right)^{\frac{1}{p-1}}+ \right.\\&\left.~~~~+ \frac{|\mu_k|(B_{4r}(x))}{r^{n-1}} \left(\fint_{B_{4r}(x)}{|\nabla u_k|^{\gamma_0}dx} \right)^{\frac{2-p}{\gamma_0}} \right]^{\Theta}.
	\end{align*}
	
	Letting $k \to \infty$ one obtains:
	\begin{align*}
		&|E_{\lambda,\varepsilon} \cap B_r(x)| \\&\le C\varepsilon^{\gamma_0\frac{1}{\Theta}}\lambda^{-\gamma_0} r^n \left[\left( \frac{|\mu|(\overline{B_{4r}(x)})}{r^{n-1}}\right)^{\frac{1}{p-1}}+\right.\\&\left.~~~~~~~~~~~~~~+\frac{|\mu|(\overline{B_{4r}(x)})}{r^{n-1}}\left(\fint_{B_{4r}(x)}{|\nabla u|^{\gamma_0}dx} \right)^{\frac{2-p}{\gamma_0}} \right]^{\gamma_0} +
	\\
		&+ C \varepsilon\lambda^{-\Theta}r^n\left[\left(\int_{B_{4r}(x)}{|\nabla u|^{\gamma_0}dx} \right)^{\frac{1}{\gamma_0}} + \left( \frac{|\mu|(\overline{B_{4r}(x)})}{r^{n-1}}\right)^{\frac{1}{p-1}}+ \right.\\&\left.~~~~+ \frac{|\mu|(\overline{B_{4r}(x)})}{r^{n-1}}\left(\fint_{B_{4r}(x)}{|\nabla u|^{\gamma_0}dx} \right)^{\frac{2-p}{\gamma_0}} \right]^{\Theta}.
	\end{align*}
	As $|x-x_1|<r$, $B_{4r}(x)\subset B_{5r}(x_1)$. This gives:
		\begin{align}\label{eq:btp1}
		\begin{split}
	\fint_{B_{4r}(x)}|\nabla u|^{\gamma_0}dx&\leq  \frac{|B_5(0)|}{|B_4(0)|} 	\fint_{B_{5r}(x_1)}|\nabla u|^{\gamma_0}dx\\&\leq C\sup_{\rho>0} \fint_{B_{\rho}(x_1)}|\nabla u|^{\gamma_0}dx
	\\&= C\mathbf{M}\left(|\nabla u|^{\gamma_0}\right)(x_1).
	\end{split}
	\end{align}
	Similarly, from $|x-x_2|<r$, we can get $B_{4r}(x)\subset B_{5r}(x_2)$ and for all $\rho>0$, it finds:
	\begin{align}\label{eq:btp2}
		\begin{split}
	\frac{|\mu|(\overline{B_{4r}(x)})}{r^{n-1}} &\le \frac{|\mu|(B_{5\rho}(x_2))}{\rho^{n-1}}\leq 5^{n-1} \mathbf{M}_1(\mu)(x_2) \le (\varepsilon^{\frac{1}{(p-1)\gamma_0}}\gamma_0)^{p-1}	.\end{split}
	\end{align}
	
	Applying \eqref{eq:btp1} and \eqref{eq:btp2} together with \eqref{eq:bt7}, \eqref{eq:bt8} yields that:	
	\begin{align*}
	|E_{\lambda,\varepsilon} \cap B_r(x)| &\le C\varepsilon^{\gamma_0  \frac{1}{\Theta}+\gamma_0  \frac{1}{(p-1)\gamma_0}} r^n(1+\varepsilon^{ \frac{1}{(p-1)\gamma_0} (p-2)})^{\gamma_0}+\\
	&+ C\varepsilon r^n(1+\varepsilon^{ \frac{1}{(p-1)\gamma_0}}+\varepsilon^{ \frac{1}{(p-1)\gamma_0}(p-1)})^{\Theta}\\&\leq C\left[\varepsilon^{\gamma_0 \frac{1}{\Theta}+\gamma_0(p-1) \frac{1}{(p-1)\gamma_0}}+\varepsilon\right] r^n \\
	& \le C\varepsilon r^n.
	\end{align*} 
	which implies \eqref{5hh2310133}, herein one remarks that the constant $C$ also depends on $T_0/r_0$.\medskip\\
		 \textbf{Case 2: $B_{4r}(x) \cap \Omega^c \neq \emptyset$}:
	Let $x_3 \in \partial\Omega$ such that $|x_3-x|=\text{dist}(x,\partial\Omega)\leq 4r$. It is not difficult to check that:
	\begin{align}
	B_{4r}(x) \subset B_{10r}(x_3).
	\end{align}
	Applying Lemma \ref{lem:estimatebound} for $u_k \in W^{1,p}_{loc}(\Omega)$ and $w_k$ the solution to:
	\begin{equation}
	\label{eq:wsol10R}
	\begin{cases}
	\div(A(x,\nabla w_k) &=0, \quad \text{in}\ \ B_{10r}(x_3)\\
	w_k &= u_k, \quad \text{on} \ \ \partial B_{10r}(x_3),
	\end{cases}
	\end{equation}
for $\mu = \mu_k$ and $B_{2R} = B_{10R}(x_3)$, one has a constant $C = C(n,p,\alpha,\beta)>0$ such that:
	\begin{align}
	\label{eq:estbound1}
	\begin{split}
	\left( \fint_{B_{10r}(x_3)}{|\nabla u_k - \nabla w_k|^{\gamma_0}dx} \right)^{\frac{1}{\gamma_0}} &\le C\left[\frac{|\mu_k|(B_{10r}(x_3))}{r^{n-1}} \right]^{\frac{1}{p-1}}+ \\ &+ C\frac{|\mu_k|(B_{10r}(x_3))}{r^{n-1}} \left(\fint_{B_{10r}(x_3)}{|\nabla u_k|^{\gamma_0}dx} \right)^{\frac{2-p}{\gamma_0}},
	\end{split}
	\end{align}
	and for all $\rho>0$ satisfies $B_{\rho}(y)\subset B_{10r}(x_3)$, following Lemma \ref{111120147**} one has
	\begin{align}
	\label{eq:estbound2}
	\begin{split}
	\left(\fint_{B_{\rho/2}(y)}|\nabla w_k|^{\Theta} dx\right)^{\frac{1}{\Theta}} &\leq C_1 \left(\fint_{B_{\rho}(y)}|\nabla w_k|^{p-1} dx\right)^{\frac{1}{p-1}}, ~~\Theta>p.
\end{split}
	\end{align}
	As a version of \eqref{eq:bt12} in the ball $B_{10r}(x_3)$ one gives:
	\begin{align}
	\label{eq:es121}
	\begin{split}
|E_{\lambda,\varepsilon}\cap B_r(x)| &\le \frac{C}{(\varepsilon^{-\frac{1}{\Theta}}\lambda)^{\gamma_0}}\left[ \int_{B_{10r}(x_3)}{|\nabla u_k - \nabla w_k|^{\gamma_0}dx} + \right.\\&\left.+ \int_{B_{10r}(x_3)}{|\nabla u - \nabla u_k|^{\gamma_0}dx}\right] +\frac{C}{(\varepsilon^{-\frac{1}{\Theta}}\lambda)^{\Theta}}\int_{B_{10r}(x_3)}{|\nabla w_k|^{\Theta}dx}.
	\end{split}
	\end{align}	 
	
	Since $B_{4r}(x)\subset B_{10r}(x_3)$, similar to \eqref{eq:bt17}, we obtain:
\begin{align}\label{eq:es122}
\begin{split}
\left(\fint_{B_{2r}(x)}|\nabla w_k|^{\Theta} dx\right)^{\frac{1}{\Theta}} &\leq C \left(\fint_{B_{10r}(x_3)}|\nabla w_k|^{p-1} dx\right)^{\frac{1}{p-1}} \\ 
&\le C \left(\fint_{B_{10r}(x_3)}{|\nabla u_k|^{\gamma_0}dx} \right)^{\frac{1}{\gamma_0}}+ \\ &+ C\left( \fint_{B_{10r}(x_3)}{|\nabla u_k - \nabla w_k|^{\gamma_0}dx} \right)^{\frac{1}{\gamma_0}}.
\end{split}
\end{align}
where, the second inequality is obtained by using Holder's inequality and for $\gamma_0>p-1$.
	
	On the ball $B_{10r}(x_3)$, using the estimates \eqref{eq:estbound1} and \eqref{eq:estbound2} with \eqref{eq:es122} from above to \eqref{eq:es121}, one obtains the following estimate:
	\begin{align*}
	 &|E_{\lambda,\varepsilon} \cap B_r(x)| \\&\le C\varepsilon^{\gamma_0\frac{1}{\Theta}}\lambda^{-\gamma_0}r^n \left[\left(\frac{|\mu_k|(B_{10r}(x_3))}{r^{n-1}}\right)^{\frac{1}{p-1}} +\frac{|\mu_k|(B_{10r}(x_3))}{r^{n-1}} \left(\fint_{B_{10r}(x_3)}{|\nabla u_k|^{\gamma_0}dx} \right)^{\frac{2-p}{\gamma_0}} \right]^{\gamma_0} +
	 \\ &+ \varepsilon^{\gamma_0\frac{1}{\Theta}}\lambda^{-\gamma_0}\int_{B_{10r}(x_3)}{|\nabla u - \nabla u_k|^{\gamma_0}dx} +\\
	 &+ C \varepsilon\lambda^{-\Theta}r^n\left[\left(\fint_{B_{10r}(x_3)}{|\nabla u_k|^{\gamma_0}dx} \right)^{\frac{1}{\gamma_0}} + \left(\frac{|\mu_k|(B_{10r}(x_3))}{r^{n-1}}\right)^{\frac{1}{p-1}}+ \right.\\&\left.~~~~+ \frac{|\mu_k|(B_{10r}(x_3))}{r^{n-1}} \left(\fint_{B_{10r}(x_3)}{|\nabla u_k|^{\gamma_0}dx} \right)^{\frac{2-p}{\gamma_0}} \right]^{\Theta}.
	 \end{align*}
	By letting $k\to \infty$, it gives
	\begin{align*}
&|E_{\lambda,\varepsilon} \cap B_r(x)| \\&\le C\varepsilon^{\gamma_0\frac{1}{\Theta}}\lambda^{-\gamma_0}r^n \left[\left(\frac{|\mu|(\overline{B_{10r}(x_3)})}{r^{n-1}}\right)^{\frac{1}{p-1}} + +\frac{|\mu|(\overline{B_{10r}(x_3)})}{r^{n-1}} \left(\fint_{B_{10r}(x_3)}{|\nabla u|^{\gamma_0}dx} \right)^{\frac{2-p}{\gamma_0}} \right]^{\gamma_0} +
\\&+ C\varepsilon\lambda^{-\Theta}r^n\left[\left(\fint_{B_{10r}(x_3)}{|\nabla u|^{\gamma_0}dx} \right)^{\frac{1}{\gamma_0}} + \left(\frac{|\mu|(\overline{B_{10r}(x_3)})}{r^{n-1}}\right)^{\frac{1}{p-1}}+ \right.\\&\left.~~~~+\frac{|\mu|(\overline{B_{10r}(x_3)})}{r^{n-1}} \left(\fint_{B_{10r}(x_3)}{|\nabla u|^{\gamma_0}dx} \right)^{\frac{2-p}{\gamma_0}} \right]^{\Theta}.
	\end{align*}
	For given $x_1, x_2$ in the previous case and the definition of $x_3$, since $\text{dist}(x,\Omega) \le 4r$, we can check that these following bounds:
	\begin{align*}
	\overline{B_{10r}(x_3)} &\subset \overline{B_{14r}(x)}\subset B_{15r}(x_1)\\
	\overline{B_{10r}(x_3)} &\subset \overline{B_{14r}(x)}\subset B_{15r}(x_2)
	\end{align*}
	and the following estimates hold
	\begin{align*}
	\frac{|\mu|(\overline{B_{10r}(x_3)})}{r^{n-1}} &\le \frac{|\mu|(B_{15r}(x_2))}{r^{n-1}} \le 15^{n-1}\mathbf{M}_1(\mu)(x_2) \le (\varepsilon^{\frac{1}{(p-1)\gamma_0}}\gamma_0)^{p-1}	.
	\end{align*}
	
	On the other hand, as $|x_3-x|=\text{dist}(x,\partial\Omega)$, one obtains
	\begin{align}
	\begin{split}
	\left(\fint_{B_{10r}(x_3)}|\nabla u|^{\gamma_0}dx\right)^{\frac{1}{\gamma_0}} &\leq \left( \frac{|B_{15}(0)|}{|B_{10}(0)|} 	\fint_{B_{15r}(x_1)}|\nabla u|^{\gamma_0}dx\right)^{\frac{1}{\gamma_0}}\\&\leq C\left( \sup_{\rho>0} \fint_{B_{\rho}(x_1)}|\nabla u|^{\gamma_0}dx\right)^{\frac{1}{\gamma_0}}
	\\&= C\left(\mathbf{M}\left(|\nabla u|^{\gamma_0}\right)(x_1)\right)^{\frac{1}{\gamma_0}}.
	\end{split}
	\end{align}
	Combining these above estimates together, one finally obtains $|E_{\lambda,\varepsilon}\cap B_r(x)| \le C\varepsilon r^n$, in which the constant $C$ also depends on $T_0/r_0$.\\
	Finally, by applying Lemma \ref{lem:mainlem} for $E = E_{\lambda,\varepsilon}$, $F = F_\lambda$, we will have 
	\begin{align}
	&\left|\{({\bf M}(|\nabla u|^{\gamma_0}))^{1/\gamma_0}>\varepsilon^{-\frac{1}{\Theta}}\lambda, (\mathbf{M}_1(\mu))^{\frac{1}{p-1}}\le \varepsilon^{\frac{1}{(p-1)\gamma_0}}\lambda \}\cap Q \right| \nonumber\\ ~~~~~~~~~~~
&~~~~~~~~~~~\qquad\leq C \varepsilon \left|\{ ({\bf M}(|\nabla u|^{\gamma_0}))^{1/\gamma_0}> \lambda\}\cap Q \right|,
	\end{align}
and the proof of theorem is complete.

\end{proof}

\begin{proof}[Proof of Theorem \ref{theolorentz_estimate}]	
	Follow Theorem \ref{theo:lambda_estimate}, there exist constants $\Theta>p$, $ C>0$, $0<\varepsilon_0<1$  and a renormalized solution $u$ to equation \eqref{eq:elliptictype} with measure data $\mu$ such that for any $\varepsilon \in (0,\varepsilon_0)$, $\lambda>0$ we have:
	\begin{align}
	\label{eq:EF}
|E_{\lambda,\varepsilon}| \le C\varepsilon |F_\lambda|.
	\end{align}

	In what follows we prove the theorem \ref{theolorentz_estimate} only for the case $s \neq \infty$, and for $s=\infty$ the proof is similar.
	
	From \eqref{eq:lorentz}, for $0<s<\infty$ and $0<q<\Theta$, \eqref{eq:EF} gives:
	\begin{align*}
	\|({\bf M}(|\nabla u|^{\gamma_0}))^{1/\gamma_0}\|_{L^{q,s}(Q)}^s&=\varepsilon^{-\frac{s}{\Theta}}q\int_{0}^{\infty}\lambda^s |\{({\bf M}(|\nabla u|^{\gamma_0}))^{1/\gamma_0}>\varepsilon^{-\frac{1}{\Theta}}\lambda\}\cap Q|^{\frac{s}{q}} d\lambda\\& \leq C \varepsilon^{-\frac{s}{\Theta}+\frac{s}{q}} \int_{0}^{\infty}\lambda^s |\{ ({\bf M}(|\nabla u|^{\gamma_0}))^{1/\gamma_0}> \lambda\}\cap Q|^{\frac{s}{q}} d\lambda+\\&~~~+C\varepsilon^{-\frac{s}{\Theta}} \int_{0}^{\infty}\lambda^s |\{(\mathbf{M}_1(\mu))^{\frac{1}{p-1}}> \varepsilon^{\frac{1}{(p-1)\gamma_0}}\lambda\}\cap Q|^{\frac{s}{q}} d\lambda\\&= C \varepsilon^{s(\frac{1}{q}-\frac{1}{\Theta})}\|({\bf M}(|\nabla u|^{\gamma_0}))^{1/\gamma_0}\|_{L^{q,s}(Q)}^s+ \\&~~~+C\varepsilon^{-\frac{s}{\Theta}-\frac{s}{(p-1)\gamma_0}} \|(\mathbf{M}_1(\mu))^{\frac{1}{p-1}}\|_{L^{q,s}(Q)}^s.
	\end{align*}
	Since $s(\frac{1}{q}-\frac{1}{\Theta})>0$, one can choose $\varepsilon \in (0,\varepsilon_0)$ such that
	\begin{align*}
	C \varepsilon^{s(\frac{1}{q}-\frac{1}{\Theta})}\leq 1/2,
	\end{align*}
	and get the complete proof:
	\begin{align*}
	\|({\bf M}(|\nabla u|^{\gamma_0}))^{1/\gamma_0}\|_{L^{q,s}(Q)}\leq C\|(\mathbf{M}_1(\mu))^{\frac{1}{p-1}}\|_{L^{q,s}(Q)}.
	\end{align*}
Similarly, the result is also obtained for $s=\infty$	.
	\end{proof}

\end{document}